\documentclass{amsart}
\usepackage{booktabs}
\usepackage{color}
\usepackage{bm}
\usepackage{tgpagella}

\usepackage{amssymb}
\usepackage{mathrsfs}
\usepackage{esint}

\usepackage{soul}

\usepackage{bbold}
\usepackage{tikz-cd}
\newtheorem{theorem}{Theorem}[section]

\newtheorem{corollary}[theorem]{Corollary}

\theoremstyle{definition}

\theoremstyle{remark}
\newtheorem{remark}[theorem]{Remark}

\numberwithin{equation}{section}

\newcommand{\1}{\mathbb 1}

\newcommand{\R}{\mathbb R}

\newcommand{\be}{\begin{equation}}
\newcommand{\ee}{\end{equation}}

\DeclareFontFamily{U}{mathx}{\hyphenchar\font45}
\DeclareFontShape{U}{mathx}{m}{n}{
      <5> <6> <7> <8> <9> <10>
      <10.95> <12> <14.4> <17.28> <20.74> <24.88>
      mathx10
      }{}
\DeclareSymbolFont{mathx}{U}{mathx}{m}{n}
\DeclareFontSubstitution{U}{mathx}{m}{n}
\DeclareMathAccent{\widecheck}{0}{mathx}{"71}
\DeclareMathAccent{\wideparen}{0}{mathx}{"75}

\newcommand{\tria}{{\mathcal T}}
\newcommand{\bbT}{\mathbb{T}}

\newcommand{\cP}{\mathcal{P}}

\newcommand{\identity}{\mathrm{Id}}

\newcommand{\Ss}{\mathscr{S}}

\newcommand{\Y}{\mathscr{Y}}
\newcommand{\ZZ}{\mathscr{Z}}

\DeclareMathOperator{\ran}{ran}

\DeclareMathOperator{\supp}{supp}

\DeclareMathOperator{\diag}{diag}

\newcommand{\cL}{\mathcal L}
\newcommand{\Lis}{\cL\mathrm{is}}
\newcommand{\cLis}{\cL\mathrm{is}_c}

\title{Operator preconditioning: the simplest case}

\date{\today}

\author{Rob Stevenson, Raymond van Veneti\"{e}}
\address{
Korteweg-de Vries Institute for Mathematics,
University of Amsterdam,
P.O. Box 94248,
1090 GE Amsterdam, The Netherlands}
\email{r.p.stevenson@uva.nl, r.vanvenetie@uva.nl}

\thanks{ 
The second author has been supported by the Netherlands Organization for Scientific Research
(NWO) under contract. no. 613.001.652}



\begin{document}
\begin{abstract}
    Using the framework of operator or Calder\'{o}n preconditioning, uniform preconditioners are constructed for elliptic operators discretized with continuous finite (or boundary) elements.
The preconditioners are constructed as the composition of an opposite order operator, discretized on the same ansatz space, and two
diagonal scaling operators.
\end{abstract}

\maketitle

\section{Introduction}
This paper deals with the construction of uniform preconditioners for negative and positive order operators, discretized by continuous piecewise polynomial trial spaces,
using the framework of `operator preconditioning'~\cite{138.26}, see also~\cite{249.15, 249.16,35.8655,138.30}.

For some $d$-dimensional closed domain (or manifold) $\Omega$ and an $s \in [0,1]$,
we consider the (fractional) Sobolev space $H^s(\Omega)$ and its dual that we denote by $H^{-s}(\Omega)$.
Let $(\Ss_\tria)_{\tria \in \bbT}$ be a family of \emph{continuous piecewise} polynomials of some fixed degree $\ell$ w.r.t.~uniformly shape regular, possibly locally refined, partitions.

Given some families of uniformly boundedly invertible operators
\begin{align*}
    A_\tria &\colon \big(\Ss_\tria, \|\cdot\|_{H^{-s}(\Omega)}\big) \to \big(\Ss_\tria, \|\cdot\|_{H^{-s}(\Omega)}\big)',\\
    B_\tria &\colon \big(\Ss_\tria, \|\cdot\|_{H^{s}(\Omega)}\big) \to  \big(\Ss_\tria, \|\cdot\|_{H^{s}(\Omega)} \big)',
\end{align*}
we are interested in constructing a \emph{preconditioner} for $A_\tria$ using operator preconditioning with $B_\tria$, \emph{and} vice versa.
To this end, we introduce a uniformly boundedly invertible operator $D_\tria \colon \big(\Ss_\tria, \|\cdot\|_{H^{-s}(\Omega)}\big) \to \big(\Ss_\tria, \|\cdot\|_{H^{s}(\Omega)}\big)'$,
yielding preconditioned systems $D_\tria^{-1}B_\tria (D_\tria')^{-1}A_\tria$ and $(D_\tria')^{-1}A_\tria D_\tria^{-1}B_\tria$ that are uniformly boundedly invertible.

In earlier research,~\cite{249.97,249.975}, we already constructed such preconditioners in a more general setting where \emph{different} ansatz spaces were used to define
$A_\tria$ and $B_\tria$.
The setting studied in the current work, however, allows for  preconditioners with a remarkably simple implementation.

A typical setting is that for some $A \colon H^{s}(\Omega) \to H^{-s}(\Omega)$ and $B \colon H^{-s}(\Omega) \to H^s(\Omega)$, both
boundedly invertible and coercive, it holds that $(A_\tria u)(v) := (Au)(v)$ and $(B_\tria u)(v) := (B u)(v)$ with $u,v \in \Ss_\tria$. An example for $s = \frac12$ is that $A$ is the Single Layer Integral
operator and $B$ is the Hypersingular Integral operator.
For this case, continuity of piecewise polynomial trial functions is required for discretizing $B$, but not for $A$, for which often discontinuous piecewise polynomials are employed. Nevertheless, when the solution of the Single Layer Integral equation is expected to be smooth, e.g., when $\Omega$ is a smooth manifold, then it is advantageous to take an ansatz space of continuous (or even smoother) functions also for $A$.

An obvious choice for $D_\tria$ would be to consider $(D_\tria u)(v) := \langle u, v\rangle_{L_2(\Omega)}$.
However, a problem becomes apparent when one considers the matrix representation $\bm{D}_\tria$ of $D_\tria$ in the standard basis: the inverse matrix $\bm{D}_\tria^{-1}$, that appears in
the preconditioned system, is densely populated. In view of application cost, this inverse matrix has to be approximated,
where it generally can be expected that, in order to obtain a uniform preconditioner, approximation errors have to decrease with a decreasing (minimal) mesh size, which will be confirmed in a numerical experiment.
To circumvent this issue, we will introduce a $D_\tria$ that has a \emph{diagonal} matrix representation, so that its inverse can be exactly evaluated.

\subsection{Notation}
In this work, by $\lambda \lesssim \mu$ we mean that $\lambda$ can be bounded by a multiple of $\mu$, independently of parameters which $\lambda$ and $\mu$ may depend on, with the sole exception of the space dimension $d$, or in the manifold case, on the parametrization of the manifold that is used to define the finite element spaces on it. Obviously, $\lambda \gtrsim \mu$ is defined as $\mu \lesssim \lambda$, and $\lambda\eqsim \mu$ as $\lambda\lesssim \mu$ and $\lambda \gtrsim \mu$.

For normed linear spaces $\Y$ and $\ZZ$, in this paper for convenience over $\R$, $\cL(\Y,\ZZ)$ will denote the space of bounded linear mappings $\Y \rightarrow \ZZ$ endowed with the operator norm $\|\cdot\|_{\cL(\Y,\ZZ)}$. The subset of invertible operators in $\cL(\Y,\ZZ)$  with inverses in $\cL(\ZZ,\Y)$
will be denoted as $\Lis(\Y,\ZZ)$.

For $\Y$ a reflexive Banach space and $C \in \cL(\Y,\Y')$ being \emph{coercive}, i.e.,
\[
\inf_{0 \neq y \in \Y} \frac{(Cy)(y)}{\|y\|^2_\Y} >0,
\]
both $C$ and $\Re(C)\!:= \!\frac{1}{2}(C+C')$ are in $\Lis(\Y,\Y')$ with
\begin{align*}
\|\Re(C)\|_{\cL(\Y,\Y')} &\leq \|C\|_{\cL(\Y,\Y')},\\
\|C^{-1}\|_{\cL(\Y',\Y)} & \leq \|\Re(C)^{-1}\|_{\cL(\Y',\Y)}=\Big(\inf_{0 \neq y \in \Y} \frac{(Cy)(y)}{\|y\|^2_\Y}\Big)^{-1}.
\end{align*}
The set of coercive $C \in \Lis(\Y,\Y')$ is denoted as $\cLis(\Y,\Y')$.
If $C   \in \cLis(\Y,\Y')$, then $C^{-1} \in \cLis(\Y',\Y)$ and $\|\Re(C^{-1})^{-1}\|_{\cL(\Y,\Y')} \leq \|C\|_{\cL(\Y,\Y')}^2 \|\Re(C)^{-1}\|_{\cL(\Y',\Y)}$.

Given a family of operators $C_i \in \Lis(\Y_i,\ZZ_i)$ ($\cLis(\Y_i,\ZZ_i)$), we will write $C_i \in \Lis(\Y_i,\ZZ_i)$ ($\cLis(\Y_i,\ZZ_i)$) uniformly in $i$, or simply `uniform', when
    \[
    \sup_{i} \max(\|C_i\|_{\cL(\Y_i,\ZZ_i)},\|C_i^{-1}\|_{\cL(\ZZ_i,\Y_i)})<\infty,
     \]
or
    \[
    \sup_{i} \max(\|C_i\|_{\cL(\Y_i,\ZZ_i)},\|\Re(C_i)^{-1}\|_{\cL(\ZZ_i,\Y_i)})<\infty.
     \]

\section{Construction of $D_\tria$ in the domain case}
For some $d$-dimensional domain $\Omega$ and an $s \in [0,1]$,
we consider the Sobolev spaces
\[
    H^s(\Omega):=[L_2(\Omega),H^1(\Omega)]_{s,2},\quad H^{-s}(\Omega):=H^s(\Omega)',
\]
which form the Gelfand triple $H^{s}(\Omega) \hookrightarrow L_2(\Omega) \simeq L_2(\Omega)' \hookrightarrow H^{-s}(\Omega)$.
\begin{remark}
In this work, for convenience we restrict ourselves to Sobolev spaces with positive smoothness index which do not incorporate homogeneous Dirichlet boundary conditions and their duals.
The proofs given below can however be extended to the setting with boundary conditions, see the arguments found in~\cite{249.97,249.975}.
\end{remark}

Let $(\tria)_{\tria \in \bbT}$ be a family of \emph{conforming} partitions of $\Omega$ into (open) \emph{uniformly shape regular} $d$-simplices.
Thanks to the conformity and the uniform shape regularity, for $d>1$ we know that neighbouring $T,T' \in \tria$, i.e.~$\overline{T} \cap \overline{T'} \neq \emptyset$, have uniformly comparable sizes.  For $d=1$, we impose this uniform `\emph{$K$-mesh property}' explicitly.

Fix $\ell > 0$. For $\tria \in \bbT$, let $\Ss_{\tria}$ denote the space of continuous piecewise polynomials of degree $\ell$ w.r.t.~$\tria$, i.e.,
\[
    \Ss_{\tria}:=\{u \in H^1(\Omega): u|_T \in \cP_\ell \,(T \in \tria)\}.
\]
Additionally, for $r \in [-1,1]$, we will write $\Ss_{\tria\!, r}$ as shorthand notation for the normed linear space $\big(\Ss_\tria, \| \cdot \|_{H^r(\Omega)}\big)$.

Denote $N_\tria$ for the set of the usual Lagrange evaluation points of $\Ss_{\tria}$, and equip the latter space with $\Phi_\tria = \{\phi_{\tria\!, \nu} : \nu \in N_\tria\}$, being the canonical nodal basis defined by $\phi_{\tria\!,\nu}(\nu') := \delta_{\nu\nu'}$ ($\nu, \nu' \in N_\tria$).
For $T \in \tria$, set $h_T := |T|^{1/d}$ and let $N_T := \overline{T} \cap N_\tria$ be the set of evaluation points in $\overline{T}$. We will omit notational dependence on $\tria$ if it is clear from the context, e.g., we will simply write $\phi_\nu$.

\subsection{Operator preconditioning}
Given some family of opposite order operators $A_\tria \in \cLis(\Ss_{\tria\!, -s}, (\Ss_{\tria\!, -s})')$ and $B_\tria \in \cLis(\Ss_{\tria\!, s}, (\Ss_{\tria\!, s})')$, both uniformly in $\tria \in \bbT$, we are interested in constructing \emph{optimal} preconditioners for both $A_\tria$  and $B_\tria$, using the idea of opposite order preconditioning (\cite{138.26}).

That is, if one has an additional family of operators $D_\tria \in \Lis(\Ss_{\tria\!, -s}, (\Ss_{\tria\!, s})')$ uniformly in $\tria \in \bbT$, then uniformly preconditioned systems for
$A_\tria$ and $B_\tria$ are given by
\begin{equation}\label{eq:precondcont}
\begin{aligned}
    D_\tria^{-1}B_\tria (D_\tria')^{-1}A_\tria \in &\Lis(\Ss_{\tria\!, -s}, \Ss_{\tria\!, -s}), \\
   (D_\tria')^{-1}A_\tria D_\tria^{-1} B_\tria \in &\Lis(\Ss_{\tria\!, s}, \Ss_{\tria\!, s}),
\end{aligned}
\end{equation}
see the following diagram:
\[
  \begin{tikzcd}
      \Ss_{\tria\!,-s}  \arrow{r}{A_\tria} & (\Ss_{\tria\!, -s})' \arrow{d}{(D_\tria')^{-1}} \\
      (\Ss_{\tria\!, s})' \arrow{u}{D_\tria^{-1}}  & \Ss_{\tria\!, s} \arrow{l}{B_\tria}
  \end{tikzcd}.
\]
In the following we shall be concerned with constructing a suitable family $D_\tria$.

\subsubsection{An obvious but unsatisfactory choice for $D_\tria$}\label{sec:mass}
An option would be to consider $(D_\tria u)(v) := \langle u, v \rangle_{L_2(\Omega)}$ $(u,v \in \Ss_\tria)$, being uniformly in $\cL(\Ss_{\tria\!, -s}, (\Ss_{\tria\!, s})')$.
For showing boundedness of its inverse, let $Q_\tria$ be the $L_2(\Omega)$-orthogonal
projector onto $\Ss_\tria$ then
\begin{align*}
    \| D_\tria^{-1}\|^{-1}_{\cL((\Ss_{\tria\!, s})', \Ss_{\tria\!, -s})}&= \inf_{0 \ne u \in \Ss_{\tria\!, -s}}\sup_{0 \ne v \in H^{s}(\Omega)} \frac{\langle u, v \rangle_{L_2(\Omega)}}{\|u\|_{H^{-s}(\Omega)} \|Q_\tria v\|_{H^s(\Omega)}} \\
                            &\geq \|Q_\tria\|^{-1}_{\cL(H^s(\Omega), H^s(\Omega))},
\end{align*}
As follows from~\cite[Prop.~2.3]{249.97}, the converse is also true, i.e., uniform boundedness of $\| D_\tria^{-1}\|_{\cL((\Ss_{\tria\!, s})', \Ss_{\tria\!, -s})}$ is actually \emph{equivalent} to uniform boundedness of $\| Q_\tria\|_{\cL(H^s(\Omega), H^s(\Omega))}$.

This uniform boundedness of $\| Q_\tria\|_{\cL(H^s(\Omega), H^s(\Omega))}$ is well-known for families of \emph{quasi-uniform}, uniformly shape regular conforming partitions of $\Omega$ into say $d$-simplices.
It has also been demonstrated for families of locally refined partitions, for $d = 2$ including those that are generated by the newest vertex bisection (NVB) algorithm, see~\cite{37.2, 75.3675, 64.595}.
On the other hand, in~\cite{19.26} a one-dimensional counterexample was presented in which the  $L_2(\Omega)$-orthogonal
projector on a family of sufficiently strongly graded, although uniform $K$ meshes, is not $H^1(\Omega)$-stable.
Thus, in any case uniform $H^1(\Omega)$-stability cannot hold without any restrictions on the grading of the meshes.

Aside from this latter theoretical shortcoming, more importantly, there is a computational problem with the current choice of $D_\tria$. The matrix representation of $D_\tria$ w.r.t.~$\Phi_\tria$ is the `mass matrix' $\bm{D}_\tria := \langle \Phi_\tria, \Phi_\tria \rangle_{L_2(\Omega)}$.
Its inverse $\bm{D}_\tria^{-1}$, appearing in the preconditioner, is densely populated, and therefore has to be approximated,
where generally the error in such approximations has to decrease with a decreasing (minimal) mesh-size in order to arrive at a uniform preconditioner.

\subsection{Constructing a practical $D_\tria$}\label{sec:lumping}
To avoid the aforementioned problems, we shall construct $D_\tria \in \Lis(\Ss_{\tria\!, -s}, (\Ss_{\tria\!, s})')$ with a diagonal matrix representation. To this end, we require
some auxiliary space $\widetilde{\Ss}_{\tria} \subset H^1(\Omega)$ equipped with a local basis $\widetilde{\Phi}_\tria$ that is $L_2(\Omega)$-biorthogonal to $\Phi_\tria$ and that has `approximation properties'.
To be precise, let
$\widetilde{\Phi}_\tria := \{ \widetilde{\phi}_\nu \in H^1(\Omega) : \nu \in N_\tria\}$ be some collection that satisfies:
\begin{equation}\label{properties}
    \begin{aligned}
        \langle \widetilde{\phi}_\nu, \phi_{\nu'}\rangle_{L_2(\Omega)}  = \delta_{\nu \nu'} \langle \1, \phi_\nu \rangle_{L_2(\Omega)},\quad &\sum_{\nu \in N_\tria} \widetilde{\phi}_\nu = \1_{\Omega}, \\
        \| \widetilde{\phi}_\nu\|_{H^k(\Omega)} \lesssim \|\phi_\nu\|_{H^k(\Omega)} \quad\big(k \in \{0,1\}\big), \quad &\supp \widetilde{\phi}_\nu \subseteq \supp \phi_\nu.\footnotemark
\end{aligned}
\end{equation}
\footnotetext{This last condition can be replaced by $\widetilde{\phi}_\nu$ having (uniformly) local support.}%
We will take $D_\tria := I_\tria' \widetilde{D}_\tria$ with $\widetilde{D}_\tria$ and $I_\tria$ being defined and analyzed in the next two theorems.

\begin{theorem}\label{thm:fortin}
    The operator $\widetilde{D}_\tria \colon \Ss_{\tria\!, -s} \to (\widetilde{\Ss}_{\tria\!, s})'$, defined by $(\widetilde{D}_\tria u)(v) := \langle u, v\rangle_{L_2(\Omega)}$,
    satisfies $\widetilde{D}_\tria \in \Lis(\Ss_{\tria\!, -s}, (\widetilde{\Ss}_{\tria\!, s})')$ uniformly in $\tria \in \bbT$.
\end{theorem}
\begin{proof}
    This proof largely follows~\cite[Sect.~3.1]{249.97}, but because here we consider a Sobolev space $H^s(\Omega)$ that does not incorporate homogeneous boundary conditions, it allows for an easier proof.

     From the assumptions~\eqref{properties}, it follows that the biorthogonal `Fortin' projector $P_\tria \colon L_2(\Omega) \to H^1(\Omega)$ onto $\widetilde{\Ss}_{\tria}$ with $\ran (\identity - P_\tria) = \Ss_{\tria}^{\perp_{L_2(\Omega)}}$ exists, and is given by
    \[
        P_\tria u = \sum_{\nu \in N_\tria} \frac{\langle u, \phi_\nu\rangle_{L_2(\Omega)}}{\langle \widetilde{\phi}_\nu, \phi_\nu \rangle_{L_2(\Omega)}} \widetilde{\phi}_\nu.
    \]
Let $T \in \tria$, by~\eqref{properties} and the fact that $\langle \1, \phi_\nu \rangle_{L_2(\Omega)} \eqsim \|\phi_\nu\|_{L_2(\Omega)}^2$, we find
  for  $k \in \{0,1\}$
    \begin{equation}\label{eq:fortin}
        \|P_\tria u\|_{H^k(T)} \lesssim \sum_{\nu \in N_T}\frac{\|\widetilde{\phi}_\nu\|_{H^k(T)}}{\|\phi_\nu\|_{L_2(\Omega)}} \|u\|_{L_2(\supp \phi_\nu)} \lesssim h_T^{-k} \|u\|_{L_2(\omega_\tria(T))},
    \end{equation}
    with $\omega_\tria(T) := \bigcup_{\{\nu \in N_T\}} \supp \phi_\nu$. This shows $\sup_{\tria \in \bbT} \|P_\tria\|_{\cL(L_2(\Omega), L_2(\Omega))} < \infty$.

     From the above inequality, and $\sum_{\nu \in N_\tria} \widetilde{\phi}_\nu = \1$, we deduce that
     \begin{align*}
         \|(\identity - P_\tria)u\|_{H^1(T)} &= \inf_{p \in \cP_0} \| (\identity - P_\tria)(u-p)\|_{H^1(T)}\\
                                             &\lesssim \inf_{p \in \cP_0} \|u - p\|_{H^1(T)} + h_T^{-1} \|u-p\|_{L_2(\omega_\tria(T))}\\
                                             &\lesssim \inf_{p \in \cP_0} h_T^{-1} \|u-p\|_{L_2(\omega_\tria(T))} + |u|_{H^1(T)}\\
                                             &\lesssim |u|_{H^1(\omega_\tria(T))},
        \end{align*}
        with the last step following from the Bramble-Hilbert lemma. We conclude that $\sup_{\tria \in \bbT} \|P_\tria\|_{\cL(H^1(\Omega), H^1(\Omega))} < \infty$, and consequently by the Riesz-Thorin interpolation theorem, that
        \[
            \sup_{\tria \in \bbT} \| P_\tria\|_{\cL(H^s(\Omega), H^s(\Omega))} < \infty.
        \]

        This latter property guarantees that $\widetilde{D}_\tria$ is uniformly boundedly invertible:
        \begin{align*}
            \| \widetilde{D}_\tria \|_{\cL(\Ss_{\tria\!, -s}, (\widetilde{\Ss}_{\tria\!, s})')} &= \sup_{0 \ne u \in \Ss_{\tria\!, -s}}  \sup_{0 \ne v \in \widetilde{\Ss}_{\tria\!, s}} \frac{\langle u, v\rangle_{L_2(\Omega)}}{\| u\|_{H^{-s}(\Omega)} \|v\|_{H^s(\Omega)}} \leq 1, \\
            \| \widetilde{D}_\tria^{-1}\|^{-1}_{\cL((\widetilde{\Ss}_{\tria\!, s})', \Ss_{\tria\!, -s})} &= \inf_{0 \ne u \in \Ss_{\tria\!, -s}}  \sup_{0 \ne v \in \widetilde{\Ss}_{\tria\!, s}}            \frac{\langle u, v \rangle_{L_2(\Omega)}}{\|u\|_{H^{-s}(\Omega)} \|v\|_{H^{s}(\Omega)}}\\
                                                                                                         &= \inf_{0 \ne u \in \Ss_{\tria\!, -s}}  \sup_{0 \ne v \in H^s(\Omega)}            \frac{\langle u, v \rangle_{L_2(\Omega)}}{\|u\|_{H^{-s}(\Omega)}\|P_\tria v\|_{H^s(\Omega)}} \\
                                                                                       &\geq \|P_\tria\|^{-1}_{\cL(H^s(\Omega), H^s(\Omega))}.\qedhere
        \end{align*}
    \end{proof}

\begin{theorem}\label{thm:bijection}
For $I_\tria \colon \Ss_{\tria\!, s} \to \widetilde{\Ss}_{\tria\!, s}$ being the bijection given by $I_\tria \phi_\nu = \widetilde{\phi}_\nu$ $(\nu \in N_\tria)$, it holds that $I_\tria \in \Lis(\Ss_{\tria\!,s}, \widetilde{\Ss}_{\tria\!,s})$ uniformly in $\tria \in \bbT$.
\end{theorem}
\begin{proof}
    Note that we may write
    \[
        I_\tria u = \sum_{\nu \in N_\tria} \frac{\langle u, \widetilde{\phi}_{\nu}\rangle_{L_2(\Omega)}}{\langle \phi_\nu, \widetilde{\phi}_\nu \rangle_{L_2(\Omega)}} \widetilde{\phi}_\nu \quad\text{and}\quad
        I^{-1}_\tria u = \sum_{\nu \in N_\tria} \frac{\langle u, \phi_{\nu}\rangle_{L_2(\Omega)}}{\langle \widetilde{\phi}_\nu, \phi_\nu \rangle_{L_2(\Omega)}} \phi_\nu.
    \]
    Equivalently to~\eqref{eq:fortin}, we see for $k \in \{0,1\}$ that
    \[
        \|I_\tria u\|_{H^k(T)} \lesssim \sum_{\nu \in N_T}\frac{\|\widetilde{\phi}_\nu\|_{H^k(T)}\|\widetilde{\phi}_\nu\|_{L_2(\Omega)}}{\|\phi_\nu\|^2_{L_2(\Omega)}} \|u\|_{L_2(\supp \phi_\nu)} \lesssim h_T^{-k} \|u\|_{L_2(\omega_\tria(T))}.
    \]
    Following the same arguments as in the proof of Theorem~\ref{thm:fortin}, using that $I_\tria \1 = \1$, then reveals that $I_\tria$ is uniformly bounded. Uniformly boundedness of $I_\tria^{-1}$ follows similarly.
\end{proof}

As announced earlier, we define $D_\tria \in \cL(\Ss_{\tria\!, -s}, (\Ss_{\tria\!, s})')$ by $D_\tria := I_\tria' \widetilde{D}_\tria$, so $(D_\tria u)(v) := \langle u, I_\tria v\rangle_{L_2(\Omega)}$ ($u,v \in \Ss_{\tria}$). Combining the previous theorems gives the following
corollary.
\begin{corollary}\label{cor}
    The operator $D_\tria$ is in $\Lis(\Ss_{\tria\!, -s}, (\Ss_{\tria\!, s})')$ uniformly in $\tria \in \bbT$.
\end{corollary}
\begin{remark}
The matrix representation of $D_\tria$ w.r.t.~$\Phi_\tria$ given by
\[
    \bm{D}_\tria = \langle \Phi_\tria, I_\tria \Phi_\tria \rangle_{L_2(\Omega)} = \diag\{ \langle \1,  \phi_\nu\rangle_{L_2(\Omega)} : \nu \in N_\tria \},
\]
which is \emph{diagonal} and therefore easily invertible. The matrix $\bm{D}_\tria$ is known as the lumped mass matrix.
\end{remark}
\begin{remark}
    The operator $D_\tria$ depends merely on the \emph{existence} of a biorthogonal basis $\widetilde{\Phi}_\tria$ that satisfies~\eqref{properties}. Indeed, this basis
    does not appear in the implementation of $\bm{D}_\tria$.
\end{remark}

A possible construction of $\widetilde{\Phi}_\tria$ can be given using techniques from~\cite{249.97}.
Consider some collection of local `bubble' functions $\Theta_\tria = \{\theta_\nu \in H^1(\Omega) : \nu \in N_\tria\}$ that satisfy:
$\big| \langle \theta_\nu, \phi_{\nu'} \rangle_{L_2(\Omega)}\big| \eqsim \delta_{\nu \nu'} \|\phi_\nu\|^2_{L_2(\Omega)}$, $\|\theta_\nu\|_{H^k(\Omega)} \lesssim \| \phi_\nu\|_{H^k(\Omega)}$ $(k\in \{0,1\})$, and $\supp \theta_\nu \subseteq \supp \phi_\nu$.
Existence of such a collection can be shown by a construction on a reference $d$-simplex, and then using an affine bijection to transfer it to general elements, see~\cite[Sect.~4.1]{249.97}.
A suitable $\widetilde{\Phi}_\tria$ that satisfies~\eqref{properties} is then given by
\[
    \widetilde{\phi}_\nu := \phi_\nu + \frac{\langle \1, \phi_\nu \rangle_{L_2(\Omega)}}{\langle \theta_\nu, \phi_\nu\rangle_{L_2(\Omega)}}\theta_\nu - \sum_{\nu' \in N_\tria} \frac{\langle \phi_\nu, \phi_{\nu'}\rangle_{L_2(\Omega)}}{\langle \theta_{\nu'}, \phi_{\nu'} \rangle_{L_2(\Omega)}} \theta_{\nu'}.
\]

We emphasize that the construction of a uniform preconditioner outlined in the subsection does require any assumptions on the mesh grading.
\subsubsection{Implementation}\label{sec:implementation}
Taking $\Phi_\tria$ as basis for both $\Ss_{\tria\!, -s}$ and $\Ss_{\tria\!, s}$, the \emph{matrix representation} of the preconditioned systems from~\eqref{eq:precondcont} read as
\[
    \bm{D}_\tria^{-1} \bm{B}_\tria \bm{D}_\tria^{-\top} \bm{A}_\tria \quad \text{and}\quad \bm{D}_\tria^{-\top}\bm{A}_\tria \bm{D}_\tria^{-1} \bm{B}_\tria,
\]
where
\begin{align*}
    \bm{A}_\tria &:= (A_\tria \Phi_\tria)(\Phi_\tria),\quad \bm{B}_\tria:=(B_\tria \Phi_\tria)(\Phi_\tria),\\
    \bm{D}_\tria &:= (D_\tria \Phi_\tria)(\Phi_\tria) = \diag\{ \langle \1,  \phi_\nu\rangle_{L_2(\Omega)} : \nu \in N_\tria \}.
\end{align*}

Alternatively, we could equip the spaces with the \emph{scaled} nodal basis $\breve{\Phi}_\tria := \bm{D}^{-\frac12}_\tria \Phi_\tria$,
so that the $L_2(\Omega)$-norm of any basis function is proportional to $1$, yielding
\begin{align*}
    \bm{\breve{A}}_\tria&:=(A_\tria \breve{\Phi}_\tria)(\breve{\Phi}_\tria)=(\bm{D}^{-\frac12}_\tria)^\top \bm{A}_\tria \bm{D}_\tria^{-\frac12},\\
\bm{\breve{B}}_\tria&:=(B_\tria \breve{\Phi}_\tria)(\breve{\Phi}_\tria)=(\bm{D}^{-\frac12}_\tria)^\top \bm{B}_\tria \bm{D}_\tria^{-\frac12},\\
    \bm{\breve{D}}_\tria&:=(D_\tria \breve{\Phi}_\tria)(\breve{\Phi}_\tria)= (\bm{D}^{-\frac12}_\tria)^\top \bm{D}_\tria \bm{D}_\tria^{-\frac12} = \bm{\identity},
\end{align*}
showing that $\bm{\breve{B}}_\tria$ is a uniform preconditioner for $\bm{\breve{A}}_\tria$ (and vice versa).
To the best of our knowledge, so far this most easy form of operator preconditioning, where the stiffness matrix of some operator
w.r.t.~some basis is preconditioned by stiffness matrix of an opposite order operator w.r.t.~the same basis, has not been shown to be optimal.

\section{Manifold case}
Let $\Gamma$ be a compact
$d$-dimensional Lipschitz, piecewise smooth manifold in $\R^{d'}$ for some $d' \geq d$ without boundary $\partial\Gamma$.
For $s \in [0,1]$, we consider the Sobolev spaces
\[
     H^s(\Gamma):=[L_2(\Gamma),H^1(\Gamma)]_{s,2},\quad H^{-s}(\Gamma) := H^s(\Gamma)'.
\]
 We assume that $\Gamma$ is given as the closure of the disjoint union of $\cup_{i=1}^p \chi_i(\Omega_i)$, with, for $1 \leq i \leq p$,
$\chi_i\colon \R^d \rightarrow \R^{d'}$ being some smooth regular parametrization, and $\Omega_i \subset \R^d$ an open polytope.
 W.l.o.g.~assuming that for $i \neq j$, $\overline{\Omega}_i \cap \overline{\Omega}_j=\emptyset$, we define
 \[
 \chi\colon \Omega:=\cup_{i=1}^p \Omega_i \rightarrow \cup_{i=1}^p \chi_i(\Omega_i) \text{ by } \chi|_{\Omega_i}=\chi_i.
 \]

Let $\bbT$ be a family of conforming partitions $\tria$ of $\Gamma$ into `panels' such that, for $1 \leq i \leq p$,  $\chi^{-1}(\tria) \cap \Omega_i$ is a uniformly shape regular conforming partition of $\Omega_i$ into $d$-simplices (that for $d=1$ satisfies a uniform $K$-mesh property).

Fix $\ell > 0$, we set
\[
\Ss_{\tria}:=\{u \in H^1(\Gamma)\colon u \circ \chi |_{\chi^{-1}(T)} \in \cP_\ell \,\,(T \in \tria)\},
\]
equipped with the canonical nodal basis $\Phi_\tria=\{\phi_\nu: \nu \in N_\tria\}$.

For construction of an operator $D_\tria \in \Lis(\Ss_{\tria\!, -s}, (\Ss_{\tria\!, s})')$ one can proceed as in the domain case.
A suitable collection $\widetilde{\Phi}_\tria$ that is $L_2(\Gamma)$-biorthogonal to $\Phi_\tria$ exists.
Moreover, the analysis from the domain case applies verbatim by only changing $\langle \cdot, \cdot \rangle_{L_2(\Omega)}$ into $\langle\cdot, \cdot, \rangle_{L_2(\Gamma)}$.
A \emph{hidden problem}, however, is that the computation of $\bm{D}_\tria = \diag \{\langle \1, \phi_\nu \rangle_{L_2(\Gamma)} : \nu \in N_\tria\}$ involves integrals over $\Gamma$ that generally have to be approximated using numerical quadrature.

In~\cite{249.97} we solved this issue by defining an additional `mesh-dependent' scalar product
\[
 \langle u,v\rangle_{\tria}:=\sum_{T \in \tria} \frac{|T|}{|\chi^{-1}(T)|} \int_{\chi^{-1}(T)} u(\chi(x))v(\chi(x))dx.
\]
This is constructed by replacing on each $\chi^{-1}(T)$, the Jacobian $|\partial \chi|$ by its average $\frac{|T|}{|\chi^{-1}(T)|}$ over $\chi^{-1}(T)$.

By considering $\widetilde{\Phi}_\tria$ that is biorthogonal to $\Phi_\tria$ with respect to $\langle \cdot, \cdot \rangle_\tria$, and the linear bijection $I_\tria$ given by $I_\tria \phi_{ \nu} = \widetilde{\phi}_{\nu}$, one
is able show that the operator $D_\tria$ defined as $(D_\tria u)(v) := \langle u, I_\tria v \rangle_\tria$ satisfies the necessary requirements.
For details we refer to~\cite{249.97}. The resulting matrix representation of $D_\tria$ w.r.t.~$\Phi_\tria$ is then given by
$\bm{D}_\tria = \diag\{ \langle \1,  \phi_\nu\rangle_{\tria} \colon \nu \in N_\tria \}$.

\section{Numerical results}
Let $\Gamma = \partial [0,1]^3 \subset \R^3$ be the two-dimensional manifold without boundary given as the boundary of the unit cube, $s = \frac12$, and $\Ss_{\tria}$ the space
of continuous piecewise polynomials of degree $\ell$ w.r.t.~a partition $\tria$. We will evaluate preconditioning of the discretized Single Layer Integral operator $A_\tria \in \cLis(\Ss_{\tria\!, -s}, (\Ss_{\tria\!, -s})')$  and an (essentially) discretized Hypersingular Integral operator $B_\tria \in \cLis(\Ss_{\tria\!, s}, (\Ss_{\tria\!, s})')$.

The Hypersingular Integral operator $\tilde B \in \cL(H^{\frac12}(\Gamma), H^{-\frac12}(\Gamma))$, is only-semi coercive, but
solving $\tilde B u=f$ for $f$ with $f(\1)=0$ is equivalent to solving $B u=f$ with
$B$ given by $(Bu)(v) = (\tilde B u)(v) + \alpha \langle u, \1 \rangle_{L_2(\Gamma)} \langle v, \1 \rangle_{L_2(\Gamma)}$, for some fixed $\alpha > 0$. This operator $B$ is in $\cLis(H^{\frac 12}(\Gamma), H^{-\frac 12}(\Gamma))$, and we shall consider discretizations $B_\tria \in \cLis(\Ss_{\tria\!, s}, (\Ss_{\tria\!, s})')$ of $B$. We found $\alpha = 0.05$ to give good results
in our examples.

Equipping both $\Ss_{\tria\!, s}$ and $\Ss_{\tria\!, -s}$ with the standard nodal basis $\Phi_\tria = \{ \phi_\nu : \nu \in N_\tria\}$, the matrix representations of the preconditioned systems from Sect.~\ref{sec:lumping} read as
\[
    \bm{D}_\tria^{-1} \bm{B}_\tria \bm{D}_\tria^{-\top} \bm{A}_\tria \quad \text{and} \quad \bm{D}_\tria^{-\top}\bm{A}_\tria \bm{D}_\tria^{-1} \bm{B}_\tria,
\]
for $\bm{D}_\tria = \diag\{ \langle \1,  \phi_\nu\rangle_{L_2(\Gamma)} \colon \nu \in N_\tria \}, \bm{A}_\tria = (A_\tria \Phi_\tria)(\Phi_\tria)$ and $\bm{B}_\tria:=(B_\tria \Phi_\tria)(\Phi_\tria)$.

We calculated (spectral) condition numbers of these preconditioned systems, where this condition number is given by $\kappa_S(\bm{X}) := \rho(\bm{X})\rho(\bm{X}^{-1})$ with $\rho(\cdot)$
denoting the spectral radius. Note that the condition numbers of the preconditioned systems coincide, i.e.,
\[
\kappa_S(\bm{D}_\tria^{-1} \bm{B}_\tria \bm{D}_\tria^{-\top} \bm{A}_\tria) = \kappa_S(\bm{D}_\tria^{-\top}\bm{A}_\tria \bm{D}_\tria^{-1} \bm{B}_\tria),
\]
so we may restrict ourselves to results for preconditioning of $\bm{A}_\tria$.

We used the BEM++ software package~\cite{249.04} to approximate the matrix representation of $\bm{A}_\tria$ and $\bm{B}_\tria$
by hierarchical matrices based on adaptive cross approximation~\cite{127.7, 19.896}.

As initial partition $\tria_\bot = \tria_1$ of $\Gamma$ we take a conforming partition consisting of $2$ triangles per side, so $12$ triangles in total, with an assignment of the newest vertices
that satisfies the so-called matching condition.  We let $\bbT$ be the sequence $\{\tria_k\}_{k \geq 1}$ where the (conforming) partition $\tria_k$ is found by applying both uniform and local refinements. To be precise,
$\tria_k$ is constructed by first applying $k$ uniform bisections to $\tria_\bot$, and then $4k$ local refinements by repeatedly applying NVB to all triangles that touch a corner of the cube.

\subsection{Comparison preconditioners}
Write $\bm{G}_\tria^{D} := \bm{D}_\tria^{-1} \bm{B}_\tria \bm{D}_\tria^{-\top}$ for the preconditioner constructed in Sect.~\ref{sec:lumping}.  We will compare this with
 the preconditioner described in Sect.~\ref{sec:mass}, for which the matrix representation is given by $\bm{G}_\tria^{M} := \bm{M}_\tria^{-1} \bm{B}_\tria \bm{M}_\tria^{-\top}$ with mass matrix $\bm{M}_\tria = \langle \Phi_\tria, \Phi_\tria\rangle_{L_2(\Gamma)}$.
 Because our partitions of the two-dimensional surface are created with NVB, we know that also the latter preconditioner provides uniformly bounded condition numbers.
  In contrast to $\bm{D}_\tria^{-1}$, the inverse $\bm{M}_\tria^{-1}$ cannot be evaluated in linear complexity. We implemented the application of $\bm{M}_\tria^{-1}$ by computing an LU-factorization of $\bm{M}_\tria$.

Table~\ref{tbl:results} compares the spectral condition numbers for the  preconditioned Single Layer systems with trial spaces given by continuous piecewise linears and those by continuous piecewise cubics.
The condition numbers $\kappa_S(\bm{G}_\tria^{D}\bm{A}_\tria)$ are uniformly bounded, but quantitatively the condition numbers $\kappa_S(\bm{G}_\tria^{M}\bm{A}_\tria)$ are better.

    \begin{table}\label{tbl:results}
    \caption{Spectral condition numbers, $\kappa_S(\bm{G}^{\circ}_\tria \bm{A}_\tria)$ for $\circ \in \{D, M\}$, of the preconditioned Single Layer system discretized on $\{\tria_k\}_{k \geq 1}$, by continuous piecewise linears $(\ell = 1)$ in the middle columns and
        discretized by continuous piecewise cubics  $(\ell = 3)$ in the right columns. Here $\bm{G}^{D}_\tria$ is the preconditioner introduced in Sect.~\ref{sec:lumping}, whereas $\bm{G}^{M}_\tria$
    is the preconditioner described in Sect.~\ref{sec:mass} whose application requires an application of $\bm{M}_\tria^{-1}$, which we implemented using an LU-factorization.}
\begin{tabular}{ll|lcc|lcc}
    \multicolumn{2}{c}{Partition $\tria$}&\multicolumn{3}{c}{Linears ($\ell = 1$)} & \multicolumn{3}{c}{Cubics ($\ell = 3$)}\vspace{2pt}\\
\toprule
    $h_{min}$&  $h_{max}$   &    dofs &$\bm{G}^{D}_\tria \bm{A}_\tria$ &   $\bm{G}^{M}_\tria \bm{A}_\tria$ & dofs & $\bm{G}^{D}_\tria \bm{A}_\tria$ &  $\bm{G}^{M}_\tria \bm{A}_\tria$ \\
\midrule
$1.4\cdot 10^{  0}$&  $1.4\cdot 10^{0}  $ & $     8$&   $ 16.2$ & $1.20$ &$   56$ & $90.5$ & $1.68$ \\
$4.4\cdot 10^{- 2}$&  $5.0\cdot 10^{-01}$ & $   218$&   $ 14.9$ & $1.91$ &$ 1946$ & $87.9$ & $2.08$ \\
$1.3\cdot 10^{- 3}$&  $3.5\cdot 10^{-01}$ & $   482$&   $ 14.7$ & $2.04$ &$ 4322$ & $86.1$ & $2.17$ \\
$4.3\cdot 10^{- 5}$&  $1.7\cdot 10^{-01}$ & $   962$&   $ 14.7$ & $2.10$ &$ 8642$ & $85.0$ & $2.21$ \\
$1.3\cdot 10^{- 6}$&  $8.8\cdot 10^{-02}$ & $  2306$&   $ 15.4$ & $2.14$ &$20738$ & $84.9$ & $2.23$ \\
$4.2\cdot 10^{- 8}$&  $4.4\cdot 10^{-02}$ & $  7106$&   $ 15.6$ & $2.16$ &$ 63938$& $84.9$ & $2.24$\\
$1.3\cdot 10^{- 9}$&  $2.2\cdot 10^{-02}$ & $ 25730$&   $ 15.8$ & $2.17$ &$231554$& $84.8$ & $2.25$\\
$4.1\cdot 10^{-11}$&  $1.1\cdot 10^{-02}$ & $ 99650$&   $ 15.8$ & $2.17$ &$896834$& $84.7$ & $2.25$\\
\bottomrule
\end{tabular}
\end{table}

\subsection{Improving the preconditioner quality}
As observed in Table~\ref{tbl:results}, the preconditioner $\bm{G}^{M}_\tria$ appears to be of superior quality, but it has unfavourable computational complexity.
It does suggest a way for improving $\bm{G}^{D}_\tria$: by replacing $\bm{D}_\tria^{-1}$ with a better approximation of $\bm{M}_\tria^{-1}$, one may
hope to improve the quality. To this end, we introduce damped (preconditioned) Richardson. Let $0 < \lambda_{-} \leq \lambda_{min}(\bm{D}_\tria^{-1} \bm{M}_\tria)$, $\lambda_{max}(\bm{D}_\tria^{-1} \bm{M}_\tria) \leq \lambda_{+}$, $\bm{R}_\tria^{(0)} := 0$ and for $k \geq 0$ define
\[
    \bm{R}_\tria^{(k+1)} := \bm{R}_\tria^{(k)} + \omega \bm{D}_\tria^{-1}(\identity - \bm{M}_\tria \bm{R}_\tria^{(k)}), \quad \omega = \frac{2}{\lambda_{-} + \lambda_{+}},
\]
being the result of $k$ Richardson iterations. Correspondingly define
\begin{equation}\label{eq:richardson}
    \bm{G}_\tria^{(k)} := \bm{R}_\tria^{(k)} \bm{B}_\tria \bm{R}_\tria^{(k)}.
\end{equation}
It follows that $\bm{G}_\tria^{(1)} = \bm{G}^{D}_\tria$ and $\lim_{k \to \infty} \bm{G}_\tria^{(k)} = \bm{G}^M_\tria$.  Although we have no proof, we
suspect that $\bm{G}_\tria^{(k)}$ provides a uniform preconditioner for $\bm{A}_\tria$ due to the fact that $\bm{R}_\tria^{(k)}$ approximates $\bm{M}_\tria^{-1}$, while preserving constant functions, being a key ingredient in the proofs of Theorems~\ref{thm:fortin} and~\ref{thm:bijection}.

Values for $\lambda_{-}$ and $\lambda_{+}$ can be found by calculating the extremal eigenvalues of the corresponding preconditioned mass matrix on a reference simplex, see e.g.~\cite{316.51}. For $\ell = 1$ this gives $\omega=\frac{2(d+2)}{d+3}$, whereas for $\ell = 3$ and $d=2$ we computed $\omega = 0.836$.

Table~\ref{tbl:richardson} compares the condition numbers $\kappa_S(\bm{G}_\tria^{(k)} \bm{A}_\tria)$ for $k \in \{2,4,6\}$. We see that a few Richardson iterations drastically improves our preconditioner, making its quality on par with that of $\bm{G}_\tria^{M}$ while having a favourable linear application cost.

Finally, to show that one cannot simply use any (iterative) method for approximating $\bm{M}_\tria^{-1}$, we consider the case where one approximates this
inverse using a Jacobi preconditioner. The resulting preconditioner is then given by
\begin{equation}\label{eq:jacobi}
    \bm{G}_\tria^J := (\diag {\bm M}_\tria)^{-1} \bm{B}_\tria (\diag {\bm M}_\tria)^{-\top}.
\end{equation}
Table~\ref{tbl:jacobi} clearly displays that this is not a uniformly bounded preconditioner, which we assume is due to the fact that $(\diag {\bm M}_\tria)^{-1}$ does
not preserve constant functions for $\ell > 1$.

\begin{table}\label{tbl:richardson}
\centering
\caption{Spectral condition numbers $\kappa_S(\bm{G}_\tria^{(k)} \bm{A}_\tria)$ with $\bm{G}_\tria^{(k)}$ the preconditioner from~\eqref{eq:richardson} that incorporates $k$ Richardson iterations. The systems are discretized by continuous piecewise linears in the left columns and discretized by continuous piecewise cubics in the right columns.  }
\begin{tabular}{lccc|lccc}
\multicolumn{4}{c}{Linears  ($\ell = 1$)} & \multicolumn{4}{c}{Cubics ($\ell = 3$)}\vspace{2pt}\\
\toprule
dofs &  $k=2$ &  $k=4$ & $k=6$ &   dofs &  $k=2$ &  $k=4$ &  $k=6$   \\
\midrule
$     8 $&$       2.26 $&$       1.29 $&$       1.22$&$    56 $&$  10.1 $&$   3.99 $&$ 2.65$\\
$   218 $&$       3.05 $&$       2.07 $&$       1.94$&$  1946 $&$  8.96 $&$   3.57 $&$ 2.52$\\
$   482 $&$       3.53 $&$       2.28 $&$       2.08$&$  4322 $&$  8.80 $&$   3.59 $&$ 2.52$\\
$   962 $&$       3.79 $&$       2.44 $&$       2.19$&$  8642 $&$  8.63 $&$   3.59 $&$ 2.52$\\
$  2306 $&$       3.98 $&$       2.52 $&$       2.24$&$ 20738 $&$  8.54 $&$   3.59 $&$ 2.52$\\
$  7106 $&$       4.18 $&$       2.57 $&$       2.27$&$  63938$&$  8.54 $&$   3.59 $&$ 2.52$\\
$ 25730 $&$       4.35 $&$       2.61 $&$       2.28$&$ 231554$&$  8.54 $&$   3.59 $&$ 2.52$\\
$ 99650 $&$       4.47 $&$       2.65 $&$       2.29$&$ 896834$&$  8.54 $&$   3.59 $&$ 2.52$\\
\bottomrule
\end{tabular}
\end{table}

\begin{table}\label{tbl:jacobi}
\caption{Spectral condition numbers $\kappa_S(\bm{G}_\tria^J \bm{A}_\tria)$ with $\bm{G}_\tria^J$ from~\eqref{eq:jacobi}, and systems discretized by continuous piecewise cubics $(\ell = 3)$.}
\begin{tabular}{rc}
\toprule
dofs &    $\bm{G}_\tria^{J} \bm{A}_\tria$\\
\midrule
$   56$&$     62.6$\\
$ 1946$&$    377.1$\\
$ 4322$&$    495.6$\\
$ 8642$&$   1016.9$\\
$20738$&$   3067.8$\\
$63938$&$  10928.3$\\
\bottomrule
\end{tabular}
\end{table}

\section{Conclusion}
Considering discretized opposite order operators $\bm{A}_\tria$ and $\bm{B}_\tria$ using the same ansatz space of continuous piecewise polynomial w.r.t.~a possibly locally refined partition $\tria$, we consider matrices $\bm{D}_\tria$ such that
$\bm{D}_\tria^{-1} \bm{B}_\tria \bm{D}_\tria^{-\top} $ is a uniform preconditioner for $\bm{A}_\tria$, and $ \bm{D}_\tria^{-\top}\bm{A}_\tria \bm{D}_\tria^{-1}$ for $\bm{B}_\tria$.
The obvious choice for $\bm{D}_\tria$ would be the mass matrix, however, it yields uniformly bounded condition numbers only under a (mild) grading assumption on the mesh, and more importantly, it has the disadvantage that its inverse is dense.
We proved that when taking $\bm{D}_\tria$ as the lumped mass matrix the condition numbers are uniformly bounded, remarkably without any gradedness assumption on the mesh, while obviously its inverse  can be applied in linear cost.

In our experiments with locally refined meshes generated by Newest Vertex Bisection, the condition numbers with $\bm{D}_\tria$ as the mass matrix are quantitatively better than those found with $\bm{D}_\tria$ as the lumped mass matrix though. Constructing $\bm{D}^{-1}_\tria$ as an approximation for the inverse mass matrix by a few preconditioned damped Richardson steps with the lumped mass matrix as a preconditioner, both the resulting matrix can be applied at linear cost and the observed condition numbers are essentially as good as with the inverse mass matrix.

\newcommand{\etalchar}[1]{$^{#1}$}

\end{document}